\documentclass[11pt,reqno,twoside]{amsart}

\usepackage{amsfonts,amsmath,amssymb}
\usepackage{mathrsfs,mathtools}
\usepackage{enumerate}
\usepackage{hyperref}
\usepackage{esint}
\usepackage{graphicx}
\usepackage{bm}
\usepackage{commath}
\usepackage{esint}
\DeclareMathAlphabet{\mathpzc}{OT1}{pzc}{m}{it}

\usepackage{placeins} 
\usepackage{flafter} 

\usepackage{tikz}

\usepackage[textsize=small]{todonotes}
\numberwithin{equation}{section}

\usepackage[dvips,bottom=1.4in,right=1in,top=1in, left=1in]{geometry}

\makeatletter
\def\eqnarray{\stepcounter{equation}\let\@currentlabel=\theequation
\global\@eqnswtrue
\tabskip\@centering\let\\=\@eqncr
$$\halign to \displaywidth\bgroup\hfil\global\@eqcnt\z@
  $\displaystyle\tabskip\z@{##}$&\global\@eqcnt\@ne
  \hfil$\displaystyle{{}##{}}$\hfil
  &\global\@eqcnt\tw@ $\displaystyle{##}$\hfil
  \tabskip\@centering&\llap{##}\tabskip\z@\cr}

\def\endeqnarray{\@@eqncr\egroup
      \global\advance\c@equation\m@ne$$\global\@ignoretrue}

\newtheorem{theorem}{Theorem}[section]

\newtheorem{corollary}[theorem]{Corollary}
\newtheorem{definition}[theorem]{Definition}

\newtheorem{lemma}[theorem]{Lemma}

\newtheorem{proposition}[theorem]{Proposition}
\newtheorem{assumption}[theorem]{Assumption}
\newtheorem{remark}[theorem]{Remark}
\numberwithin{equation}{section}

\def\Omc{\mathbb{R}^N\setminus\Omega}

\def\RR{{\mathbb{R}}}

\def\Om{\Omega}

\def\bOm{\overline{\Om}}
\def\pOm{\partial\Omega}

\title[Optimal Control of Fractional Elliptic PDEs with State Constraints]{Optimal Control of Fractional Elliptic PDEs with State Constraints
and Characterization of the dual of Fractional Order Sobolev Spaces}

\author{Harbir Antil}
\address{Department of Mathematical Sciences, George Mason University, Fairfax, VA 22030, USA.}
\email{hantil@gmu.edu}

\author{Deepanshu Verma}
\address{Department of Mathematical Sciences, George Mason University, Fairfax, VA 22030, USA.}
\email{dverma2@gmu.edu}

\author{Mahamadi Warma}
\address{University of Puerto Rico, Rio Piedras Campus, Department of Mathematics,  College of Natural Sciences,  17 University AVE. STE 1701  San Juan PR 00925-2537 (USA). }
\email{mahamadi.warma1@upr.edu, mjwarma@gmail.com}

\thanks{The first and second authors are partially supported by NSF grant DMS-1818772 and 
the Air Force Office of Scientific Research under Award NO: FA9550-19-1-0036. The third 
author is partially supported by the Air Force Office of Scientific Research under 
Award NO:  FA9550-18-1-0242.}

\keywords{Optimal control with PDE constraints, state and control constraints, Fractional Laplacian, measure valued datum, characterization of fractional dual space, regularity of solutions to state and adjoint equations.}

\subjclass[2010]{49J20, 49K20, 35S15, 65R20, 65N30}

\begin{document}

\begin{abstract}
This paper introduces the notion of state constraints for optimal control problems 
governed by fractional elliptic PDEs of order $s \in (0,1)$. There are several 
mathematical tools that are developed during the process to study this problem, for instance, 
the characterization of the dual of the fractional order Sobolev spaces and well-posedness of 
fractional PDEs with measure-valued datum. These tools are widely applicable. We show 
well-posedness of the optimal control problem and derive the first order optimality conditions. 
Notice that the adjoint equation is a fractional PDE with measure as the right-hand-side datum. We 
use the characterization of the fractional order dual spaces to study the regularity of the
state and adjoint equations. We emphasize that the classical case ($s=1$) was considered by 
E. Casas in \cite{ECasas_1986a} but almost none of the existing results are applicable to our fractional
case. 
\end{abstract}

\maketitle

\section{Introduction}

Let $\Om \subset \mathbb{R}^N$ ($N \ge 1$) be a bounded open set with boundary $\pOm$. 
The main goal of this paper is to introduce and study the following state and control 
constrained optimal control problem:
\begin{subequations}\label{eq:dcp}
 \begin{equation}\label{eq:Jd}
    \min_{(u,z)\in (U,Z)} J(u,z) 
 \end{equation}
 subject to the fractional elliptic PDE: Find $u \in U$ solving 
 \begin{equation}\label{eq:Sd}
 \begin{cases}
    (-\Delta)^s u &= z \quad \mbox{in } \Om, \\
                u &= 0 \quad \mbox{in } \RR^N\setminus \Om  ,
 \end{cases}                
 \end{equation} 
with the state constraints 
 \begin{equation}\label{eq:Ud}
    u|_{\Omega} \in \mathcal{K} := \left\{ w \in C_0(\Omega) \ : \ w(x) \le u_b(x) , 
        \quad \forall x \in \overline{\Om}  \right\} 
 \end{equation}
where $C_0(\Om)$ is the space of continuous functions in $\overline\Om$ that vanishes
on $\pOm$ and $u_b \in C(\overline\Om)$. 
Extending functions by zero outside $\Omega$, we can then identity $C_0(\Omega)$ with the space $\{ u\in C_c(\RR^N):\; u=0\;\mbox{ in }\;\RR^N\setminus\Omega\}$.

Moreover, we assume the control constraints
 \begin{equation}\label{eq:Zd}
    z \in Z_{ad} \subset 
    L^p(\Om) 
 \end{equation}
with $Z_{ad}$ being a non-empty, closed, and convex set and the real number $p$ satisfies
 \end{subequations}
\begin{equation}\label{cond-p}
\begin{cases}
p>\frac{N}{2s}\;\;&\mbox{ if }\; N>2s,\\
p>1 \;\;&\mbox{ if }\; N=2s,\\
p=1\;\;&\mbox{ if }\; N<2s.
\end{cases}
\end{equation} 
Notice that for $z \in L^p(\Om)$, with $p$ as given in \eqref{cond-p}, we have that 
$u \in L^\infty(\Om)$, see \cite{antil2017optimal} for details. 

Optimal control of fractional PDEs with control constraints has recently received a lot 
of attention. We refer to \cite{antil2017optimal} for the optimal control of fractional 
semilinear PDEs with both spectral and integral fractional Laplacians with distributed control, 
see also  \cite{MDElia_MGunzburger_2014a} for such a control of an integral operator. We refer to 
\cite{HAntil_JPfefferer_SRogovs_2018a} for the boundary control with spectral fractional Laplacian 
and \cite{HAntil_RKhatri_MWarma_2018a, antil2019external} for the exterior optimal control of 
fractional PDEs. See \cite{HAntil_MWarma_2019a,antil2018b} for the optimal control of quasi-linear
fractional PDEs where the control lies in the coefficient.

We remark that the case $s=1$ is classical 
see for instance \cite{ECasas_1986a}, we also refer to \cite{ECasas_1993a,ECasas_1997a}, see also
\cite{ECasas_MMateos_BVexler_2014a} for more recent results. We also refer to  
to the monographs \cite{MHinze_RPinnau_MUlbrich_SUlbrich_2009a,FTroeltzsch_2010a} and the references therein.
Nevertheless, none of these existing works are directly applicable to the case of fractional state 
constraint as stated in \eqref{eq:dcp}. 

The key difficulties in studying \eqref{eq:dcp}--\eqref{cond-p} and 
the novelties of this paper are outlined next. 
    \begin{itemize}
        \item {\bf Nonlocal equation.} The equation \eqref{eq:Sd} is nonlocal, see 
              Section~\ref{s:not} for the precise definition of the nonlocal operator 
              $(-\Delta)^s$.
              
        \item {\bf Continuity of the state solution.} Similarly to the classical case, we need to 
              show that the solution $u$ to \eqref{eq:Sd} is continuous whenever $z \in L^p(\Om)$. 
              Recall that for such a $z$, that the solution $u\in L^\infty(\Om)$ is due to our previous work 
              \cite{antil2017optimal}. Our 
              continuity result in this paper, in a sense, weakens the regularity requirements on 
              $z$ in comparison to the celebrated result of 
              \cite[Proposition 1.1]{XRosOton_JSerra_2014a} where the authors assumed that 
              $z \in L^\infty(\Om)$. 
              
         \item {\bf Equation with measure valued data.} The adjoint equation is a fractional
              PDE with measure-valued datum. We shall first show the well-posedness of such PDEs
              in the space $L^{p'}(\Om)$ where $p$ is as in \eqref{cond-p}.                   
              
        \item {\bf Characterization of the dual space $\widetilde{W}^{-s,p'}$.} Let $1\le p<\infty$, $p'=\frac{p}{p-1}$ and let $\widetilde{W}^{-s,p'}$ denote the dual of $\widetilde{W}_0^{s,p}$ (see Section \ref{s:not}).
        Recall that the 
              classical dual space, $W^{-1,p}(\Om)$, of $W_0^{1,p}(\Om)$ can be characterized 
              in terms of vector-valued 
              $L^p(\Om)$-spaces \cite[Theorem~3.9]{RAAdams_JJFFournier_2003a}. 
              Such a characterization of the space $\widetilde{W}^{-s,p'}$ is essential to study the regularity 
              of the aforementioned adjoint equation (fractional PDE with measure-valued datum)
              and the state equation with weaker than $L^p(\Om)$ datum. 
              However, to the best of our knowledge, this characterization has remained open for 
              the fractional order Sobolev spaces. This characterization obtained here is one of the main novelty of the current paper.        
              
   \item {\bf Higher regularity of solutions to the Dirichlet problem \eqref{eq:Sd}.}    Using the above characterization of  the dual spaces of the fractional order Sobolev spaces, we have shown that for $z\in \widetilde{W}^{-t,p}$, for appropriate $p$ and $0<t<1$, solutions of the Dirichlet problem  \eqref{eq:Sd} are also continuous up to the boundary of $\Omega$. This is the first time that such a regularity result has been proved  (with very weak right-hand side) for the fractional Laplace operator.
    \end{itemize}
    
We recognize that the fractional operators are starting to play a pivotal in several applications: 
imaging science, 
phase field models, 
Magnetotellurics in geophysics, 
electrical response in cardiac tissue,  
diffusion of biological species, 
and data science, see \cite{antil2019external} and references therein.
In fact under a very general setting, the article \cite{AGrigoryan_TKumagai_2008a} shows that there 
are only two types of heat kernels: diffusion (exponential), or heat kernels for $s$-stable 
processes (polynomial). Notice that the fractional Laplace operator is the generator of  the $s$-stable L\'evy
process. For a general description of nonlocal/fractional heat kernels and their relationship 
to stochastic processes, we refer to \cite{chen2017heat, MR1974415}. 

The rest of the paper is organized as follows. In Section~\ref{s:not}, we first introduce
the underlying notation and state some preliminary results. These results are well-known.
Our main work starts from Section~\ref{s:sa} where we first establish the continuity of 
solution to the state equation. In addition, we establish the well-posedness of the fractional
PDEs with measured valued datum. In Section~\ref{eq:ocp}, we show the well-posedness of the
control problem and derive the optimality conditions. In Section~\ref{s:char} we derive
the characterization of dual spaces of fractional order Sobolev spaces. We conclude this paper by giving a higher regularity result for the associated adjoint equation in Section \ref{regularity}.

\section{Notation and preliminaries}
\label{s:not}

We begin this section by introducing some notation and preliminary results. 
We follow the notation from our previous works \cite{antil2018b,HAntil_RKhatri_MWarma_2018a}. Unless 
otherwise stated, $\Omega \subset \RR^N$ $(N\ge 1)$ is a bounded open set, $0 < s < 1$ 
and $1 \le p < \infty$. For a sufficiently regular function $u$ defined on $\RR^N$, we shall denote by $D_{s,p} u$ the function 
defined on $\RR^N \times\RR^N$ by 
    \[
        D_{s,p} u[x,y] :=\frac{u(x)-u(y)}{|x-y|^{{\frac{N}{p}}+s}} .
    \]
Then we define the Sobolev space  
    \[
        W^{s,p}(\Om) := \Big\{ u \in L^p(\Om) \ : \ D_{s,p} u \in 
                L^p(\Om \times \Om)  \Big\} 
    \]
which we endow with the norm  
 \[
    \|u\|_{W^{s,p}(\Om)} := \left(\int_\Om |u|^p\;dx 
        + \|D_{s,p} u\|_{L^p(\Om\times\Om)}^p \right)^{\frac1p}.      
 \]
 
We let
\begin{equation*}
 W_{0}^{s,p}(\Omega ):=\overline{\mathcal{D}(\Omega )}^{W^{s,p}(\Omega )} ,
\end{equation*}
where $\mathcal{D}(\Om)$ denotes the space of smooth functions with compact support in $\Omega$. 

We have taken the following result from \cite[Theorem 1.4.2.4, p.25]{Gris} (see also \cite{Chen,War}).
\begin{theorem}\label{theo-gris}
Let $\Omega\subset\RR^N$ be a bounded open set with a Lipschitz continuous boundary and $1<p<\infty$. 
Then the following assertions hold.
\begin{enumerate}
\item[(a)] If $0<s\le\frac 1p$, then $W^{s,p}(\Omega )= W_0^{s,p}(\Omega)$. 
\item[(b)] If $\frac 1p<s<1$, then $W_0^{s,p}(\Omega)$ is a proper closed subspace of $W^{s,p}(\Omega)$.
\end{enumerate}
\end{theorem}

Since $\Omega$ is assumed to be bounded, we have the following continuous embedding:
\begin{equation}\label{inj1}
W_0^{s,2}(\Omega)\hookrightarrow
\begin{cases}
L^{\frac{2N}{N-2s}}(\Omega)\;\;&\mbox{ if }\; N>2s,\\
L^p(\Omega),\;\;p\in[1,\infty)\;\;&\mbox{ if }\; N=2s,\\
C^{0,s-\frac{N}{2}}(\bOm)\;\;&\mbox{ if }\; N<2s.
\end{cases}
\end{equation}

We shall let
\begin{align*}
2^\star:=\frac{2N}{N-2s},\;\;N>2s.
\end{align*}

Using potential theory, a complete characterization of $W_0^{s,p}(\Omega)$ for arbitrary bounded open sets 
has been given in \cite{War}. Notice that from Theorem~\ref{theo-gris} it follows that for a bounded open
set with Lipschitz boundary, if $\frac{1}{p} < s < 1$, then 
\begin{equation}
\Vert u\Vert_{W_{0}^{s,p}(\Om)}=\|D_{s,p} u\|_{L^p(\Om\times\Om)}
\label{norm-26}
\end{equation}%
defines an equivalent norm on $W_{0}^{s,p}(\Om)$. We shall always use this norm for the space 
$W_{0}^{s,p}(\Omega)$.


In order to study the fractional Laplace equation \eqref{eq:Sd} we need to consider 
the following function space 
\begin{align*}
\widetilde{W}_0^{s,p}(\Om):=\Big\{u\in W^{s,p}(\RR^N):\;u=0\;\mbox{ on }\;\RR^N\setminus\Omega\Big\}.
\end{align*}
Let $\Omega\subset\RR^N$ be a bounded open set with a Lipschitz continuous boundary. It has been shown in \cite[Theorem 6]{Val} that $\mathcal D(\Omega)$ is dense in $\widetilde{W}_0^{s,p}(\Om)$. Moreover, for every $0<s<1$, we have
\begin{align}
\|u\|_{\widetilde{W}_0^{s,p}(\Om)}^p:=&\|D_{s,p} u\|_{L^p(\RR^N\times\RR^N)}^p=\int_{\RR^N}\int_{\RR^N}\frac{|u(x)-u(y)|^p}{|x-y|^{N+sp}}\;dxdy \nonumber \\
=& \int_{\Omega}\int_{\Omega}\frac{|u(x)-u(y)|^p}{|x-y|^{N+sp}}\;dxdy +2\int_{\Omega}|u(x)|^p\int_{\RR^N\setminus\Omega}\frac{1}{|x-y|^{N+sp}}\;dy\;dx \nonumber \\
=&\|D_{s,p} u\|_{L^p(\Omega\times\Omega)}^p+\int_{\Omega}|u|^p\kappa(x)\;dx,
 \label{eq:eqnorm}
\end{align}
where
\begin{align*}
\kappa(x)=2\int_{\RR^N\setminus\Omega}\frac{1}{|x-y|^{N+sp}}\;dy.
\end{align*}

\begin{remark}\label{remark}
{\em We mention the following observations.
\begin{enumerate}
\item The embedding \eqref{inj1} holds with $W_0^{s,2}(\Om)$ replaced by $\widetilde{W}_0^{s,2}(\Om)$.

\item Let $p$ satisfy \eqref{cond-p} and $p':=\frac{p}{p-1}$. We claim that $\widetilde{W}_0^{s,2}(\Om)\hookrightarrow L^{p'}(\Omega)$. Indeed, we have the following three cases.
\begin{itemize}
\item If $N>2s$, that is $p>\frac{N}{2s}$, then $p'<\frac{N}{N-2s}<\frac{2N}{N-2s}$. In this case we have that $\widetilde{W}_0^{s,2}(\Om)\hookrightarrow L^{\frac{2N}{N-2s}}(\Omega) \hookrightarrow L^{p'}(\Omega)$, where we have used \eqref{inj1} and the fact that $\Omega$ is bounded.

\item If $N=2s$, that is $p\in (1,\infty)$ is any a number , then $\widetilde{W}_0^{s,2}(\Om)\hookrightarrow L^{p'}(\Omega)$ by \eqref{inj1}.

\item If $N<2s$, that is $p=1$, then $p'=\infty$ and $\widetilde{W}_0^{s,2}(\Om)\hookrightarrow L^{p'}(\Omega)$ by \eqref{inj1}.
\end{itemize}

\end{enumerate}
}
\end{remark}
We next state an important result for 
$\widetilde{W}_0^{s,p}(\Om)$  (recall Theorem~\ref{theo-gris} for $W_0^{s,p}(\Om)$). For a proof, 
we refer to \cite[Theorem~2.3]{antil2018b}. 

\begin{theorem}
\label{thm:gris2}
Let $\Omega\subset\RR^N$ be a bounded open set with a Lipschitz continuous boundary and $1<p<\infty$. 
If $\frac 1p<s<1$, then $\widetilde{W}_0^{s,p}(\Om)=W_0^{s,p}(\Omega)$ with equivalent norms. 
\end{theorem}

Thus from Theorem~\ref{thm:gris2} it follows that for a bounded open set with Lipschitz 
boundary, 
if $\frac{1}{p} < s < 1$, then 
\begin{equation}
\label{eq:blah}
\Vert u\Vert_{\widetilde{W}_{0}^{s,p}(\Om)}=\|D_{s,p} u\|_{L^p(\Om\times\Om)} ,
\end{equation}%
in other words the second term in \eqref{eq:eqnorm} is not relevant. 

If $0<s<1$, $p\in (1,\infty )$ and $p^{\prime }:=\frac{p}{p-1}$, then  the space $
\widetilde{W}^{-s,p^{\prime }}(\Om)$ is defined as the dual of $\widetilde{W}_{0}^{s,p}(\Om )$, 
i.e., $ \widetilde{W}^{-s,p^{\prime }}(\Om ):=(\widetilde{W}_{0}^{s,p}(\Om ))^{\star }$.
We notice that a characterization of this dual space given in Section~\ref{s:char} is one of the novelties of this 
paper. 

After all these preparations, we are now ready to define the fractional Laplacian. 
We set 
\begin{equation*}
\mathbb{L}_s^{1}(\RR^N):=\left\{u:\RR^N\rightarrow
\mathbb{R}\;\mbox{
measurable, }\;\int_{\RR^N}\frac{|u(x)|}{(1+|x|)^{N+2s}}%
\;dx<\infty \right\}.
\end{equation*}%
For $u\in \mathbb{L}_s^{1}(\RR^N)$ and $\varepsilon >0$, we let
\begin{equation*}
(-\Delta )_{\varepsilon }^{s}u(x)=C_{N,s}\int_{\{y\in \RR^N,|y-x|>\varepsilon \}}
\frac{u(x)-u(y)}{|x-y|^{N+2s}}dy,\;\;x\in\RR^N,
\end{equation*}%
where $C_{N,s}$ is a normalization constant and it is given by
\begin{equation}\label{CN}
C_{N,s}:=\frac{s2^{2s}\Gamma\left(\frac{2s+N}{2}\right)}{\pi^{\frac
N2}\Gamma(1-s)},
\end{equation}%
and $\Gamma $ is the standard Euler Gamma function (see, e.g. \cite{Caf3,NPV,War-DN1,War}). 
We then define the {\bf fractional Laplacian} 
$(-\Delta )^{s}$ for $u\in \mathbb{L}_s^{1}(\RR^N)$  by the formula
\begin{align}
(-\Delta )^{s}u(x)=C_{N,s}\mbox{P.V.}\int_{\RR^N}\frac{u(x)-u(y)}{|x-y|^{N+2s}}dy 
=\lim_{\varepsilon \downarrow 0}(-\Delta )_{\varepsilon
}^{s}u(x),\;\;x\in\RR^N,\label{eq11}
\end{align}
provided that the limit exists. Notice that  \cite[Proposition 2.2]{BPS} shows that 
for $u \in \mathcal{D}(\Om)$, we have 
 \[
    \lim_{s\uparrow 1^-}\int_{\RR^N} u (-\Delta)^su\;dx 
        = \int_{\RR^N} |\nabla u|^2 dx 
        = - \int_{\RR^N} u \Delta u\;dx 
        = - \int_{\Om} u \Delta u\;dx.
 \]
This limit makes use of the constant $C_{N,s}$. 

We define the operator $(-\Delta)_D^s$ in $L^2(\Omega)$ as follows
\begin{equation}\label{eq:DeltasD}
D((-\Delta)_D^s)=\Big\{u\in \widetilde{W}_0^{s,2}(\Om):\; (-\Delta)^su\in L^2(\Omega)\Big\},\;\;(-\Delta)_D^s(u|_{\Om})=(-\Delta)^su\;\mbox{ in }\;\Omega.
\end{equation} 
Notice that $(-\Delta)_D^s$ is the realization in $L^2(\Om)$ of the fractional Laplace operator 
$(-\Delta)^s$ with the Dirichlet exterior condition $u=0$ in $\RR^N\setminus\Omega$. We refer to \cite{ClWa} for a rigorous definition of $(-\Delta)_D^s$.

Finally, we close this section by recalling the integration-by-parts formula for $(-\Delta)^s$ 
(see e.g. \cite{SDipierro_XRosOton_EValdinoci_2017a}). 
 \begin{proposition}[\bf The integration by parts formula for $(-\Delta)^s$]
 \label{prop:prop}
 Let $u \in \widetilde{W}_0^{s,2}(\Om)$ be such that $(-\Delta)^su \in L^2(\Omega)$. 
 Then for every $v\in \widetilde{W}_0^{s,2}(\Om)$ we have 
      \begin{align}\label{Int-Part}
       \frac{C_{N,s}}{2} 
        \int_{\RR^N}\int_{\RR^N}
         \frac{(u(x)-u(y))(v(x)-v(y))}{|x-y|^{N+2s}} \;dxdy 
        = \int_\Om v(-\Delta)^s u\;dx . 
      \end{align}
 \end{proposition}

\section{State and adjoint equations}\label{s:sa}

Throughout the remainder of the paper, given a Banach space $X$ and its dual $X^\star$, we shall denote by $\langle \cdot,\cdot\rangle_{X^\star,X}$ their duality pairing.

The purpose of this section is to show that the weak solutions to \eqref{eq:Sd} are continuous
and to study the existence and uniqueness of $L^{p'}(\Om)$-solutions to the system
 \begin{equation}\label{eq:Sda}
 \begin{cases}
    (-\Delta)^s u &= \mu \quad \mbox{in } \Om, \\
                u &= 0 \quad \mbox{in } \RR^N\setminus \Om  ,
 \end{cases}                
 \end{equation} 
where $\mu \in \mathcal{M}(\Omega)$. Here $\mathcal{M}(\Om)$ denotes the space of all
Radon measures on $\Om$. More precisely, $\mathcal{M}(\Om) = C_0(\Om)^*$ i.e.,
$\mathcal{M}(\Om)$ is the dual of $C_0(\Om)$ such that 
    \[
        \langle \mu , v \rangle_{(C_0(\Om))^\star,C_0(\Om)} = \int_\Om v\;d\mu , 
        \quad \mu \in \mathcal{M}(\Om), \quad v \in C_0(\Om) .
    \]
In addition, we have the following norm on this space:
    \[
        \|\mu\|_{\mathcal{M}(\Om)} = \sup_{v \in C_0(\Om),|v|\le 1} \int_\Om v\;d\mu . 
    \]    

We will first show the continuity of weak solutions 
to \eqref{eq:Sd}. We recall that the paper \cite{XRosOton_JSerra_2014a} proves the optimal 
H\"older $C^s$-regularity of $u$ under the condition that the datum $z \in L^\infty(\Omega)$. 
However, in our setting we have only assumed that $z \in L^p(\Om)$, therefore the result
of \cite{XRosOton_JSerra_2014a} does not apply. Before we recall the results 
from \cite{antil2017optimal} and \cite{XRosOton_JSerra_2014a}, respectively, we state
the notion of weak solution to \eqref{eq:Sd}. 

 \begin{definition}[\bf Weak solution to Dirichlet problem] 
 \label{def:weak_d}
    Let $z \in \widetilde{W}^{-s,2}(\Om)$. A function $u \in \widetilde{W}_0^{s,2}(\Om)$ is said 
    to be a weak solution to \eqref{eq:Sd} if the identity
    \begin{equation*}\label{eq:vw_d}
      \frac{C_{N,s}}{2} \int_{\RR^N} \int_{\RR^N}   
          \frac{(u(x)-u(y))(v(x)-v(y))}{|x-y|^{N+2s}} \; dxdy 
         = \langle z, v\rangle , \quad \forall v \in \widetilde{W}_0^{s,2}(\Om),
    \end{equation*}
     holds.
  \end{definition} 

    \begin{proposition}
    \label{prop:Antil_Warma}
        Let $\Omega$ be a bounded Lipschitz domain. Assume that $z\in L^p(\Omega)$ with
        $p$ as in \eqref{cond-p}. Then every weak solution $u$ of  \eqref{eq:Sd} 
        belongs to $L^\infty(\Omega)$ and there is a constant $C=C(N,s,p,\Omega)>0$ such that        
        \begin{align}\label{inf-norm}
            \|u\|_{L^\infty(\Omega)}\le C\|z\|_{L^p(\Omega)}.
        \end{align}    
    \end{proposition}   
In Theorem~\ref{Main2}, we shall reduce the $L^p(\Om)$ regularity requirement on the datum $z$ given  in Proposition \ref{prop:Antil_Warma}.
    \begin{proposition}
    \label{solbound}
        Let $\Omega$ be a bounded Lipschitz domain satisfying the exterior cone condition.
        Assume that $z\in L^{\infty} (\Omega)$. Then every weak solution $u$ of \eqref{eq:Sd} 
        belongs to $C^s(\RR^N)$ and there is a constant $C=C(N,s,p,\Omega)>0$ such that  
        \begin{equation}
        \label{Cs-norm}
            \| u\|_{ C^{s}(\RR^N)} \leq C \| z\|_{L^{\infty}(\Omega)} . 
        \end{equation}
    \end{proposition}

After giving the above two results, we are ready to state the first main result of this section. 

    \begin{theorem} 
    \label{pbound}
        Let $\Omega$ be a bounded Lipschitz domain satisfying the exterior cone condition. 
        Assume that $z\in L^p(\Omega)$ with $p$ as in \eqref{cond-p}. Then every weak solution 
        $u$ of \eqref{eq:Sd} belongs to $C_0(\Om)$ and there is a constant $C=C(N,s,p,\Omega)>0$ 
        such that 
        \begin{equation}
             \| u\|_{ C_0(\Om)} \leq C \| z\|_{L^{p}(\Omega)} . 
        \end{equation}
    \end{theorem}
    
    \begin{proof}
    Since $L^{\infty}(\Omega)$ is dense in $L^p(\Omega)$, therefore given $z \in L^p(\Om)$,
    we can construct a sequence $\{ z_n \}_{n \in \mathbb{N}} \subset L^{\infty}(\Omega)$ such 
    that     
    \[ 
        \|z_n - z\|_{L^p(\Omega)}\rightarrow 0 \;\;\mbox{ as }\; n\to\infty. 
    \]    

    For each $n \in \mathbb{N}$, let $u_n$ solve
    \begin{align}\label{eq:Sddd}
        \begin{cases}
        (-\Delta)^s u_n &= z_n \quad \mbox{in } \Om, \\
                    u_n &= 0 \quad \mbox{in } \RR^N\setminus \Om  .
        \end{cases}
    \end{align}
    Then from Proposition~\ref{solbound}, we have that $u_n \in C^s(\RR^N)$. 
    Next, subtracting \eqref{eq:Sd} from \eqref{eq:Sddd}, we deduce that 
    \begin{equation*}
        \begin{cases}
        (-\Delta)^s (u_n-u) &= z_n-z \quad \mbox{in } \Om, \\
                    (u_n-u) &= 0 \quad \mbox{in } \RR^N\setminus \Om  . 
        \end{cases}
    \end{equation*}    
    Since $(z_n - z) \in L^p(\Om)$, we can apply Proposition~\ref{prop:Antil_Warma}
    to deduce that 
    \[
        \| u_n-u \|_{L^{\infty}(\Omega)} \leq C \| z_n - z \|_{L^p(\Omega)} 
        \rightarrow 0 \quad \mbox{as } n \rightarrow \infty . 
    \]
    Thus 
    \[
        \| u_n-u \|_{L^{\infty}(\Omega)} \rightarrow 0 \quad \mbox{as } n \rightarrow \infty . 
    \]
    Since $u_n \in C_0(\Om)$, it follows that $u \in C_0(\Om)$ and the proof is complete. 
    \end{proof}
We shall reduce the $L^p(\Om)$ regularity requirement on the datum $z$ in the above result in Corollary~\ref{cor:Main2}. 

Towards this end, we introduce the notion of very-weak solution to \eqref{eq:Sda} with measure
valued right-hand-side datum. We refer to \cite{HAntil_RKhatri_MWarma_2018a,antil2019external} for the notion of 
very-weak solution with $L^2(\Omc)$ exterior datum. 

 \begin{definition}[\bf Very-weak solution to the Dirichlet problem with measure datum] 
 \label{def:vweak_d}
 Let $p$ be as in \eqref{cond-p} and $\frac{1}{p}+\frac{1}{p'} = 1$. 
    Let $\mu \in \mathcal{M}(\Omega)$.
    A function $u \in L^{p'}(\Om)$ is said to be a very-weak solution 
    to \eqref{eq:Sda} if the identity
    \begin{equation*}
      \int_{\Om} 
          u(-\Delta)^s v\;dx
         = \int_\Om v \; d\mu ,
    \end{equation*}
   holds for every $v \in V := \{v \in C_0(\Om) \cap \widetilde{W}_0^{s,2}(\Om) \;:\; (-\Delta)^s v \in L^p(\Om) \}$. 
 \end{definition}

In the next theorem we present the second main result of this section. 


    \begin{theorem}\label{thm:mesdata}
        Let $\Omega$ be a bounded Lipschitz domain satisfying the exterior cone condition. 
        Let $\mu \in \mathcal{M}(\Om)$ and $p$ as in \eqref{cond-p}.  Then there exists a unique $u \in L^{p'}(\Omega)$,
        with $p'$ such that $\frac{1}{p} + \frac{1}{p'} = 1$, that solves \eqref{eq:Sda} 
        according to the Definition~\ref{def:vweak_d} and there is a constant $C=C(N,s,p,\Omega)>0$ 
        such that 
        \[
            \|u\|_{L^{p'}(\Om)} \le C \|\mu\|_{\mathcal{M}(\Om)} . 
        \]
    \end{theorem}
    
    \begin{proof}
        For a given $\xi \in L^p(\Om)$ we begin by considering the following auxiliary problem
        \begin{equation}\label{eqv}
            \begin{cases}
            (-\Delta)^s v &= \xi \quad \mbox{in } \Om, \\
                        v &= 0 \quad \mbox{in } \RR^N\setminus \Om  .
            \end{cases}
        \end{equation}
        Since $L^p(\Om) \hookrightarrow L^{p'}(\Om) \hookrightarrow\widetilde{W}^{-s,2}(\Om)$ (by using Remark \ref{remark}), with the embedding being
        continuous, it follows that there exists  a unique $v \in \widetilde{W}_0^{s,2}(\Om)$ satisfying \eqref{eqv}.         
        Then according to Theorem~\ref{pbound}, we have that $v \in C_0(\Om)$.       
        Towards this end we define the mapping 
        \[
            \begin{aligned}
                \Xi : L^p(\Om) &\rightarrow C_0(\Om) \\
                       \xi &\mapsto \Xi \xi := v .  
            \end{aligned}
        \]    
        Notice that $\Xi$ is linear and continuous (due to Theorem~\ref{pbound}). 
        
        Let us define $u := \Xi^* \mu$. Then $u\in L^{p'}(\Om)$. We shall show that $u$ solves \eqref{eq:Sda}.         
        Notice that
        \begin{equation}\label{eq:p'soln}
            \int_\Om u\xi \;dx = \int_\Om u (-\Delta)^s v \; dx 
            = \int_\Om (\Xi^*\mu) \xi \; dx = \int_\Om v \;d\mu .
        \end{equation}
        Thus, we have constructed a unique function $u$ that solves \eqref{eq:Sda} according to 
        the Definition~\ref{def:vweak_d}. It then remains to prove the required bound. From 
        \eqref{eq:p'soln} we have that 
        \begin{equation}\label{eq:imp}
            \left| \int_\Om u\xi \;dx \right| 
                \le \|\mu\|_{\mathcal{M}(\Om)} \|v\|_{C_0(\Om)}
                \le C \|\mu\|_{\mathcal{M}(\Om)} \|\xi\|_{L^p(\Om)},
        \end{equation}
        where in the last step we have used Theorem~\ref{pbound}. Then dividing both sides 
        by $\|\xi\|_{L^p(\Om)}$ and taking the supremum over $\xi \in L^{p}(\Om)$ we obtain the desired result. The proof is finished.
    \end{proof}
The regularity of $u$, given in Theorem~\ref{thm:mesdata} , and solving \eqref{eq:Sda} will be improved in Corollary~\ref{cor:Main3}.

\section{Optimal control problem} \label{eq:ocp}

Throughout this section, we will operate under the conditions of Theorem~\ref{thm:mesdata}.    

The purpose of this section is to study the existence of solution to the optimal control problem \eqref{eq:dcp} and establish the first order optimality conditions. 

We begin by rewriting the optimal control problem \eqref{eq:dcp}. 
We recall from \eqref{eq:DeltasD} that $(-\Delta)^s_D$ is the realization in $L^2(\Om)$
of the fractional Laplacian $(-\Delta)^s$ which incorporates zero exterior Dirichlet
condition and it is a self-adjoint operator on $L^2(\Om)$. 
As a result, the problem \eqref{eq:dcp} can be rewritten as
	\begin{equation} 
	\begin{aligned}
		&\min_{(u,z)\in (U,Z)} J(u,z) \\ 
		\text{subject to}& \\
		& (-\Delta)_D^s u = z , \quad \mbox{in } \Om \\
		&u|_{\Om} \in \mathcal{K} \quad \mbox{and} \quad 
		z \in Z_{ad}  .
	\end{aligned}	
	\end{equation}

Next, we introduce relevant function spaces. We let 
    \begin{equation*}
    \begin{aligned}
        Z &:= L^p(\Om) ,  \quad \mbox{with $p$ as in \eqref{cond-p} but } 1 < p < \infty , \\
        U &:= 
            \{u\in \widetilde{W}_0^{s,2}(\Om)\cap C_0(\Om): (-\Delta)^s_D (u|_\Om) \in L^{p}(\Om)\}.
    \end{aligned}    
    \end{equation*}
Then $U$ is a Banach space with the graph norm 
$\|u\|_{U}:=\|u\|_{\widetilde{W}_0^{s,2}(\Om)} + \|u\|_{C_0(\Om)} 
+ \|(-\Delta)^s_D (u|_\Om)\|_{L^p(\Om)}$. 
Notice that $(-\Delta)^s_D (u|_\Om) = (-\Delta)^s u$ in $\Om$. 
We let $Z_{ad} \subset Z$ be a nonempty, closed, and convex set and 
$\mathcal{K}$ as in \eqref{eq:Ud}, i.e., 
    \begin{equation}\label{eq:Ud_1}
        u|_{\Om} \in \mathcal{K} := \left\{ w \in C_0(\Omega) \ : \ w(x) \le u_b(x) , 
            \quad \forall x \in \overline{\Om}  \right\}. 
    \end{equation}

Notice that for every $z \in Z$, due to Theorem~\ref{pbound}, there is a 
unique $u \in U$ that solves the state equation \eqref{eq:Sd}. Using this fact, 
the control-to-state (solution) map
    \[
        S : Z \rightarrow U,\;\;\;z \mapsto Sz =: u 
    \]       
is well-defined, linear, and continuous. Since $U$ is continuously embedded into
$C_0(\Om)$, then we can consider the control-to-state map as 
    \[
        E \circ S : Z \rightarrow  C_0(\Om) .
    \]

Towards this end, we define the admissible control set as
    \[
        \widehat{Z}_{ad} := \left\{ z \in Z \; : \; 
            z \in Z_{ad}, \ (E \circ S) z \in \mathcal{K} \right\} ,
    \]
and as a result, the reduced minimization problem is given by 
     \begin{equation}\label{eq:rpDir}
        \min_{z \in \widehat{Z}_{ad}} \mathcal{J}(z) := J( (E \circ S) z,z) . 
    \end{equation}

%
Next, we state the well-posedness result for \eqref{eq:dcp} and equivalently 
\eqref{eq:rpDir}. 
\begin{theorem}\label{thm:docexist}
Let $Z_{ad}$ be a bounded, closed, and convex subset of $Z$ and $\mathcal{K}$ be a convex and
closed subset of $C_0(\Om)$ such that $\widehat{Z}_{ad}$ is nonempty. If 
$J : L^2(\Om) \times L^p(\Om)$ is weakly lower-semicontinuous, then there is a 
solution to \eqref{eq:rpDir}. 
\end{theorem}
\begin{proof}
The proof is based on the so-called direct method or the Weierstrass theorem 
\cite[Theorem~3.2.1]{HAttouch_GButtazzo_GMichaille_2014a}. We will provide
some details for completeness. We can always construct a minimizing sequence 
$\{z_n\}_{n=1}^\infty \subset Z$ such that 
$\inf_{z \in Z_{ad}} \mathcal{J}(z) = \lim_{n\rightarrow\infty} \mathcal{J}(z_n)$. 
Since $Z_{ad}$ is bounded, it follows that $\{z_n\}_{n=1}^\infty $  is a
bounded sequence. Due to the reflexivity of $Z$, there exists a weakly convergent subsequence
$\{z_n\}_{n=1}^\infty$ (not relabeled) such that $z_n \rightharpoonup \bar{z}$  in $Z$ as
$n\rightarrow\infty$. Next, due to $Z_{ad}$ being closed and 
convex, thus weakly closed, we obtain that $\bar{z} \in Z_{ad}$. 

Notice that $C_0(\Om)$ is non-reflexive. However, we have that 
$u_n = S z_n \in U \hookrightarrow C_0(\Om)$ and $S \in \mathcal{L}(Z,C_0(\Om))$, therefore 
we have a subsequence $\{u_n\}$ (not-relabeled) that converges weakly$^\star$ to $\bar{u}$ in $C_0(\Om)$. 
Since $\mathcal{K}$ is also weakly closed, we have that $\bar{u} \in \mathcal{K}$. 

Owing to the uniqueness of the limit and the assumption that $\widehat{Z}_{ad}$ is nonempty
we can deduce that $\bar{z} \in \widehat{Z}_{ad}$. Finally, it remains to show that $\bar{z}$
is a solution to \eqref{eq:rpDir}. This follows from the weak lower-semicontinuity
assumption on $J$. 
%
\end{proof}

Next, we derive the first order necessary optimality conditions, but before 
we make the following standard assumption.

    \begin{assumption}[\bf Slater condition]\label{ass:data_compatibility}
        There is some control function $\widehat{z} \in Z_{ad}$ 
        such that the corresponding state $u$ fulfills the strict state constraint 
    \begin{equation}\label{eq:Robinson}
        u(x) < u_b(x) \quad \forall x \in \overline\Omega . 
    \end{equation}
    \end{assumption}   
See the monographs \cite{MHinze_RPinnau_MUlbrich_SUlbrich_2009a, FTroeltzsch_2010a} for a further discussion.    


Using the definition of $U$ we have that $(-\Delta)^s_D : U \mapsto Z$ is a bounded
operator and from Theorem~\ref{pbound} it is a surjective operator. 
We have the following
first order necessary optimality conditions:
%
\begin{theorem}\label{thm:optcond}
    Let $J : L^2(\Om) \times L^p(\Om) \rightarrow \mathbb{R}$ be continuously Fr\'echet differentiable 
    and assume that \eqref{eq:Robinson} holds. Let $(\bar{u},\bar{z})$ be a solution to the optimization
    problem \eqref{eq:dcp}. Then there exist Lagrange multipliers 
    $\bar{\mu} \in (C_0(\Om))^\star$ and an adjoint variable $\bar{\xi} \in L^{p'}(\Om)$ such that 
    \begin{subequations}
        \begin{align}
            (-\Delta)^s_D \bar{u} &= \bar{z} , \quad \mbox{in } \Om , \label{eq:a}  \\             
                  \langle \bar{\xi},(-\Delta)^s_D v \rangle_{L^{p'}(\Om),L^p(\Om)} 
                &= \left( J_u(\bar{u},\bar{z}),v \right)_{L^2(\Om)}  
                   + \int_\Om v \; d\bar{\mu}    
                   , && \forall \;v \in U \label{eq:b} \\
                \langle \bar{\xi} + J_z(\bar{u},\bar{z}) , 
                    z - \bar{z} \rangle_{L^{p'}(\Om),L^p(\Om)} 
                 &\ge 0 , && \forall \;z \in Z_{ad}  \label{eq:c}  \\
             \bar{\mu} \ge 0, \quad  \bar{u}(x) \le u_b(x) \mbox{ in } \Om, 
             \quad &\mbox{and} \quad \int_{\Om} (u_b - \bar{u})\;d\mu = 0 \label{eq:d} .      
        \end{align}
    \end{subequations}
\end{theorem}

\begin{proof}
We begin by checking the requirements for 
\cite[Lemma~1.14]{MHinze_RPinnau_MUlbrich_SUlbrich_2009a}. We notice that 
$(-\Delta)^s_D : U \mapsto Z$ is bounded and surjective. Moreover, the condition
\eqref{eq:Robinson} implies that the interior of the set $\mathcal{K}$, is 
nonempty. It then remains to show the existence of a $(\hat{u},\hat{z}) 
\in U \times Z_{ad}$ such that 
    \begin{equation}\label{eq:1}
        (-\Delta)^s_D (\hat{u}-\bar{u}) - (\hat{z}-\bar{z}) = 0  \;\;\mbox{ in }\;\Omega. 
    \end{equation}
Since $(\bar{u},\bar{z})$ solves the state equation, therefore from \eqref{eq:1}
we have that 
    \begin{equation}\label{eq:2}
        (-\Delta)^s_D \hat{u} = \hat{z} \;\;\mbox{ in }\;\Omega.
    \end{equation}
Notice that for every $\hat{z} \in Z_{ad}$, there is a unique $\hat{u}$ that
solves \eqref{eq:2}, in particular $(\hat{u},\hat{z})$ works. Thus we immediately
obtain \eqref{eq:a}--\eqref{eq:c}. Instead of \eqref{eq:d} we obtain that 
    \begin{equation}\label{eq:3}
        \bar{\mu} \in \mathcal{K}^\circ , \quad  u(x) \le u_b(x), \quad x \in \Om, 
        \quad \mbox{and} \quad
            \langle \bar{\mu},\bar{u}\rangle_{C_0(\Om)^*,C_0(\Om)} = 0 ,
    \end{equation}
where $\mathcal{K}^\circ$ denotes the polar cone. Then the equivalence between 
\eqref{eq:3} and \eqref{eq:d} follows from a classical result in functional 
analysis, see \cite[pp.~88]{MHinze_RPinnau_MUlbrich_SUlbrich_2009a} for details.
\end{proof}

\section{Characterization of the dual of fractional order Sobolev spaces}
\label{s:char}

Given $0<s<1$, $1\le p<\infty$ and $p'=\frac{p}{p-1}$, the aim of this section is to give a complete characterization of the space $\widetilde W^{-s,p'}(\Omega)$. Recall that $\widetilde W^{-s,p'}(\Omega)$ is the dual of the space $\widetilde W_0^{s,p}(\Omega)$. 

We start by stating this abstract result taken from \cite[pp.~194]{HWAlt_1992a}.

\begin{lemma} \label{cong}
If $X$ and $W$ are two Banach spaces, then $X\times W$ is also a Banach space with the associated norm 
    \[
        \| (x,y)\|_{X\times W} = \|x\|_{X} + \|y\|_{W}  . 
    \]     
Moreover, the dual of the product space, $(X\times W)^\star$ is isometrically isomorphic to the product of the dual spaces, that is, 
    \[ 
        (X\times W)^* \cong X^\star \times W^\star . 
    \] 
\end{lemma}

Let $1\le p<\infty$ and let $Y:=L^p(\Om) \times L^p(\RR^N \times \RR^N)$ be endowed with the  norm 
    \[ 
        \|(v_1,v_2)\|_{Y}
          =\left( \|v_1\|_{L^p}^p + \|v_2\|_{L^p(\RR^N \times \RR^N)}^p \right)^{\frac{1}{p}}  . 
    \]
For $v\in \widetilde{W}_0^{s,p}(\Om)$, we associate the vector $Pv \in Y$ given by 
\begin{align} \label{isom}
 Pv=(v,D_{s,p}v).
\end{align}
Since 
$\|Pv\|_{Y}=\|(v,D_{s,p}v)\|_{Y}=\|v\|_{\widetilde{W}_0^{s,p}(\Om)}$, we have that  $P$ is an 
isometry and hence, injective, so $P:\widetilde{W}_0^{s,p}(\Om)\mapsto Y$ is an isometric isomorphism of 
$\widetilde{W}_0^{s,p}(\Om)$ onto its image $Z\subset Y$. Also, $Z$ is a closed subspace of $Y$, 
because $\widetilde{W}_0^{s,p}(\Om)$ is complete (isometries preserve completion).

Throughout this section without any mention, we shall let
\begin{align*}
Y:=L^p(\Om) \times L^p(\RR^N \times \RR^N).
\end{align*}

\begin{lemma}\label{lemma:ref1}
 Let $1\le p < \infty$. Then for every $f\in Y^\star$, there exists a unique $u=(u_1,u_2)\in L^{p'}(\Om) \times L^{p'}(\RR^N \times \RR^N)$  such that for every $v=(v_1,v_2)\in Y$, we have 
    \[
        f(v)=\langle u_1,v_1 \rangle_{L^{p'}(\Om),L^p(\Om)} + \langle u_2,v_2 \rangle_{L^{p'}(\RR^N \times \RR^N),L^p(\RR^N \times \RR^N)}  . 
    \]
Moreover, 
    \[ 
        \| f \|_{Y^*} = \| u \|_{L^{p'}(\Om) \times L^{p'}(\RR^N \times \RR^N)}= \| u_1 \|_{L^{p'}(\Om) } + \| u_2 \|_{L^{p'}(\RR^N \times \RR^N)}  .
    \]
\end{lemma}

\begin{proof}
    Let $w\in L^{p}(\Om)$. Then $(w,0)\in Y$. We define $f_1(w):=f(w,0)$. Then 
    $f_1\in (L^{p}(\Om))^\star$. For arbitrary $w_1,w_2 \in L^{p}(\Om)$ and for scalars 
    $\alpha, \beta$, we have 
    \begin{align*}
        f_1(\alpha w_1+\beta w_2)&=f(\alpha w_1+\beta w_2,0)
          =f(\alpha (w_1,0)+\beta (w_2,0)) \\ 
         &=\alpha f((w_1,0))+\beta f((w_2,0)) \\
         &=\alpha f_1( w_1)+\beta f_1(w_2) ,
    \end{align*}
    and for all $w \in L^p(\Om)$, we have that 
    \begin{align*}
        |f_1(w)|=|f((w,0))|\leq \|f\|_{Y^\star} \|(w,0)\|_{Y}
         =\|f\|_{Y^\star} \|w\|_{L^{p}(\Om)}  .
    \end{align*}
    Thus $f_1 \in (L^p(\Om))^\star$. Notice that $(L^p(\Om))^\star = L^{p'}(\Om)$. 
    
    Similarly, let $w\in L^{p}(\RR^N \times \RR^N)$ then $(0,w)\in Y$. Thus for 
    $w\in L^{p}(\RR^N \times \RR^N)$, if we define $f_2(w):=f(0,w)$, then 
    $f_2\in (L^{p}(\RR^N \times \RR^N))^\star$. Notice that 
    $(L^{p}(\RR^N \times \RR^N))^\star = L^{p'}(\RR^N \times \RR^N)$. 

    Hence, by the Riesz Representation theorem there exist a unique $u_1\in L^{p'}(\Om)$ 
    and a unique $u_2\in L^{p'}(\RR^N \times \RR^N)$ such that 
    \[ 
        f(v_1,0)=f_1(v_1) = \langle u_1,v_1 \rangle_{L^{p'}(\Om),L^p(\Om)}  \quad  \forall \; v_1 \in L^{p}(\Om) 
    \] 
    and 
    \[ 
        f(0,v_2) = f_2(v_2) = \langle u_2,v_2 \rangle_{L^{p'}(\RR^N \times \RR^N),L^p(\RR^N \times \RR^N)}  \quad  \forall\; v_2\in L^{p}(\RR^N \times \RR^N) . 
    \]
    
    Now let $v := (v_1,v_2)$. Notice that $v$ is an arbitrary element of $Y$ and 
    we can write $v=(v_1,v_2)=(v_1,0)+(0,v_2)$. Hence 
    \begin{align*}
        f(v)=f(v_1,0)+f(0,v_2)=f_1(v_1)+f_2(v_2)
            =  \langle u_1,v_1 \rangle_{L^{p'}(\Om),L^p(\Om)} + \langle u_2,v_2 \rangle_{L^{p'}(\RR^N \times \RR^N),L^p(\RR^N \times \RR^N)} . 
    \end{align*}
    Moreover, 
    \begin{align*}
        |f(v)|& \leq \|u_1\|_{L^{p'}(\Om)} \| v_1 \|_{L^{p}(\Om)}  
            + \|u_2\|_{L^{p'}(\RR^N \times \RR^N)} \|v_2\|_{L^{p}(\RR^N \times \RR^N)} \\
        & \leq \|{u}\|_{L^{p'}(\Om) \times L^{p'}(\RR^N \times \RR^N)} \|{v}\|_{Y}.
    \end{align*}
    Therefore, 
    \begin{equation}\label{eq:neq}
        \|f\|_{Y^\star} \leq \|u\|_{L^{p'}(\Om) \times L^{p'}(\RR^N \times \RR^N)}  . 
    \end{equation}
    The proof of the first part is complete. It then remains to show that the norms in 
    \eqref{eq:neq} are equal. 
    
    Let us first consider the case $1 < p < \infty$. Define 
    \[ 
        v_1(x)=
        \begin{cases} 
            |u_1(x)|^{p'-2}\overline{u_1(x)}, & \text{if } u_1(x)\neq 0 ,\\ 
                0 ,  & \text{if } u_1(x)=0 ,
        \end{cases} 
    \] 
    and  
    \[ 
        v_2(x,y)=
        \begin{cases} 
            |u_2(x,y)|^{p'-2}\overline{u_2(x,y)}, & \text{if } u_2(x,y)\neq 0, \\ 
             0, &         \text{if } u_2(x,y)=0 .
        \end{cases} 
    \]
    Then, for $v=(v_1,v_2)$, we have that 
    \begin{align*}
        |f(v)|&=|f(v_1,v_2)|=|f((v_1,0)+(0,v_2))|=|f_1(v_1)+f_2(v_2)| \\
        & =\left|\ \langle u_1,v_1 \rangle_{L^{p'}(\Om),L^p(\Om)} + \langle u_2,v_2 \rangle_{L^{p'}(\RR^N \times \RR^N),L^p(\RR^N \times \RR^N)} \right |\\ 
          &=\|u_1\|_{L^{p'}(\Om)}^{p'} + \| u_2 \|_{L^{p'}(\RR^N \times \RR^N)}^{p'} \\
        &=\|u\|_{L^{p'}(\Om)\times L^{p'}(\RR^N \times \RR^N)}^{p'}
         =|\langle u,v \rangle_{Y^\star,Y}|  
         = \| v \|_{Y} \| u \|_{Y^\star} 
         = \| v \|_{Y} \| u \|_{L^{p'}(\Om)\times L^{p'}(\RR^N \times \RR^N)},
    \end{align*}
where we have used the equality in H\"older's inequality, the equality holds because $|v_i|^p=|u_i|^{p'}$. Moreover, we have used the fact that $ Y^\star \cong L^{p'}(\Om)\times L^{p'}(\RR^N \times \RR^N)$ due to Lemma \ref{cong}. 

Let us consider the case $p=1$, then $Y = L^1(\Om) \times L^1(\RR^N\times \RR^N)$ and 
we can set (due to Lemma~\ref{cong}) $Y^\star = L^\infty(\Om) \times L^\infty(\RR^N\times \RR^N)$. Notice that 
$\|u \|_{Y^\star} :=\max \left\{\|u_1 \|_{L^\infty(\Om)}, \|u_2\|_{L^\infty(\RR^N\times \RR^N)} \right\}$. 
It is sufficient to show that $\|f\|_{Y^\star} \ge \|u\|_{Y^\star}$ to get the desired result.
Now for any $\epsilon > 0$ and $k=1$ there exists a measurable set $A\subset \Omega$ 
(or $\subset \RR^N \times \RR^N$ when $k=2$) with finite, non zero measure such that 
$|u_k(x)|\geq \|u \|_{Y^\star}-\epsilon$, 
$\forall \; x\in A$.

Next, we define 
    \[ 
        v_k(x)=
            \begin{cases} 
                \frac{\overline{u_k(x)}}{|u_k(x)|} & \text{for  } x\in A\; \text{and } 
                u_k(x)\neq 0, \\ 0 & \text{elsewhere }  . 
            \end{cases} 
    \]
Set $v=(v_k,0)$ if $k=1$, otherwise set $v=(0,v_k)$. Then,
\[
    |f(v)|=|\langle u_k,v_k \rangle_{Y^\star,Y}|=\int_{A} |u_k(x)| \, dx 
    \geq \left( \|u \|_{Y^\star}-\epsilon  \right)\|v\|_{Y} . 
\] 
Since $\epsilon$ is chosen arbitrarily, the result follows from the definition of operator norm. 
\end{proof}

\begin{theorem}
\label{thm:chartW}
Let $1\le p<\infty$. Assume $f\in \widetilde{W}^{-s,p'}(\Om)$. Then there exists 
$(f^0,f^1) \in L^{p'}(\Om)\times L^{p'}(\RR^N \times \RR^N)$ such that 
    \begin{equation}\label{char2}
        \langle f,v \rangle_{\widetilde{W}^{-s,p'}(\Om), \widetilde{W}^{s,p}_0(\Om)} 
            = \int_{\Omega} f^0 v \, dx + \int_{\RR^N} \int_{\RR^N} f^1 D_{s,p} v[x,y] 
                    \; dx \; dy, \quad \forall\; v\in \widetilde{W}^{s,p}_0(\Om) ,
    \end{equation}
and
    \begin{align}\label{norm2}        
        \|f\|_{\widetilde{W}^{-s,p'}(\Om)} 
            = \inf \left\{ \|(f^0,f^1)\|_{L^{p'}(\Om)\times L^{p'}(\RR^N \times\RR^N)}\right\}     
    \end{align} 
where the infimum is taken over all $(f^0,f^1) \in L^{p'}(\Om)\times L^{p'}(\RR^N \times \RR^N)$ for which \eqref{char2} holds for every $v\in \widetilde{W}^{s,p}_0(\Om)$. Moreover, if $1 < p < \infty$ then $(f^0,f^1) $ is unique. 
\end{theorem}
\begin{proof}
Define the linear functional $\widehat{L} : Z \rightarrow \mathbb{R}$, where $Z \subset Y$ is 
the range of $P$ given in  \eqref{isom}, by 
    \[ 
        \widehat{L}(Pv)=f(v), \quad v\in \widetilde{W}_0^{s,p}(\Om) 
    \]
\begin{center}
\begin{tikzpicture}
    \node (E) at (0,0) {$\widetilde{W}_0^{s,p}(\Om) $};
    \node[right=of E] (F) {$Z\subset Y$};
    \node (R) at (1.2,-1.3) {$\RR$};
    \draw[->] (E)--(F) node [midway,above] {$P$};
    \draw[->] (F)--(R) node [midway,right] {$\widehat{L}$} ;
    \draw[->] (E)--(R) node [midway,left] {$f$} ;
\end{tikzpicture}
\end{center}
Since $P$ is an isometric isomorphism onto $Z$, it follows that $\widehat{L}\in Z^\star$ and
\begin{align*}
    \| \widehat{L} \|_{Z^\star} 
        = \sup_{\|Pv\|_{Y}=1}| \langle \widehat{L},Pv \rangle_{Y^\star,Y}|
        =\sup_{\|v\|_{\widetilde{W}_0^{s,p}}=1}|\langle f,v \rangle_{\widetilde{W}^{-s,p'}(\Om), \widetilde{W}^{s,p}_0(\Om)} |
        =\|f\|_{W^{-s,p'}(\Om)} . 
\end{align*} 
Then, by the Hahn-Banach extension theorem, there exists an 
$L\in Y^\star = L^{p'}(\Om)\times L^{p'}(\RR^N \times \RR^N)$ such that 
$\|L\|_{Y^\star}=\|\widehat{L}\|_{Z^\star}$. Since $L \in Y^\star$, using Lemma \ref{lemma:ref1},
there exists $(f^0,f^1)\in L^{p'}(\Om)\times L^{p'}(\RR^N \times \RR^N)$ such that 
for $v=(v_1,v_2)\in Y$, we have
    \[ 
        L(v)=\langle f^0,v_1 \rangle_{L^{p'}(\Om),L^{p}(\Om)} + \langle f^1,v_2 \rangle_{L^{p'}(\RR^N \times \RR^N),L^{p}(\RR^N \times \RR^N)} = \int_{\Omega} f^0 v_1 \; dx + \int_{\RR^N} \int_{\RR^N} f^1  v_2 \; dx  dy . 
    \]
Notice that when $1 < p < \infty$, $(f^0,f^1)$ is unique due to the uniform convexity of the Banach space $L^p(\Om) \times L^p(\RR^N\times \RR^N)$. 
    
Thus, for $v\in \widetilde{W}_0^{s,p}(\Om)$ we have $Pv\in Y$. Then using the definition of 
$\widehat{L}$ we get
    \[ 
        f(v)=\widehat{L}(Pv)= L(Pv)=L(v,D_{s,p}v)=\int_{\Omega} f^0 v \, dx + \int_{\RR^N} \int_{\RR^N} f^1  D_{s,p}v \; dx  dy 
    \] 
which is \eqref{char2} after noticing that $ \langle f,v \rangle_{\widetilde{W}^{-s,p'}(\Om), \widetilde{W}^{s,p}_0(\Om)}  = f(v)$. Moreover, we have 
    \[ 
        \|f\|_{\widetilde{W}^{-s,p'}(\Om)}=\|\widehat{L}\|_{Z^\star}
             =\|L\|_{Y^\star}
             =\|(f^0,f^1)\|_{L^{p'}(\Om)\times L^{p'}(\RR^N \times \RR^N)} .
    \]
The proof for the case $1 < p < \infty$ is complete.     
    
Now, for arbitrary $(g^0,g^1)\in L^{p'}(\Om)\times L^{p'}(\RR^N \times \RR^N)$, for which \eqref{char2} holds for all 
$v\in \widetilde{W}_0^{s,p}(\Om)$, we can define $L_g$ as 
    \[
        L_g(u) = \langle g^0, u_1 \rangle_{L^{p'}(\Om),L^p(\Om)} 
         + \langle g^1, u_2 \rangle_{L^{p'}(\RR^N\times\RR^N),L^{p}(\RR^N\times\RR^N)}, \quad 
         \forall \;u \in Y . 
    \]
Then $L_g \in Y^\star$ and $L_g|_Z = \widehat{L}$ (due to \eqref{char2}). As a result, 
    \[
        \|\widehat{L}\|_{Z^\star} \le \|L_g\|_{Y^\star} .
    \]
Thus 
    \[
        \|f\|_{\widetilde{W}^{-s,p'}(\Om)} \le \|g\|_{L^{p'}(\Om)\times L^{p'}(\RR^N\times \RR^N)} . 
    \]    
The proof is complete. 
\end{proof}

In view of Theorem~\ref{thm:gris2}(b), for $\frac{1}{p} < s < 1$ we can everywhere replace 
$\RR^N \times \RR^N$ in Theorem~\ref{thm:chartW} by $\Om \times \Om$. More precisely, we have the following result.

\begin{corollary}\label{cor-dual}
Let $1 < p < \infty$ and $\frac{1}{p} < s < 1$. Let $f\in \widetilde{W}^{-s,p'}(\Om)$. Then there exists a unique $(f^0,f^1) \in L^{p'}(\Om)\times L^{p'}(\Om \times \Om)$ such that 
    \begin{equation}\label{char21}
         \langle f,v \rangle_{\widetilde{W}^{-s,p'}(\Om), \widetilde{W}^{s,p}_0(\Om)} 
            = \int_{\Omega} f^0 v \, dx + \int_{\Om} \int_{\Om} f^1 D_{s,p} v[x,y] 
                    \; dx \; dy, \quad \forall\; v\in \widetilde{W}^{s,p}_0(\Om) ,
    \end{equation}
and
    \begin{align}\label{norm21}        
        \|f\|_{\widetilde{W}^{-s,p'}(\Om)} 
            = \inf \left\{ \|(f^0,f^1)\|_{L^{p'}(\Om)\times L^{p'}(\Om \times\Om)}\right\}     
    \end{align} 
where the infimum is taken over all $(f^0,f^1) \in L^{p'}(\Om)\times L^{p'}(\Om \times \Om)$ for which \eqref{char2} holds for every $v\in \widetilde{W}^{s,p}_0(\Om)$.
\end{corollary}

\section{Improved regularity of state and higher regularity of adjoint}\label{regularity}

%

In this section, we study the higher regularity properties of solutions to the  Dirichlet problem \eqref{eq:Sd}
with right hand side $z\in \widetilde W^{-t,p}(\Omega)$ for some suitable $p\in (1,\infty)$ and $0<t<1$.

Throughout the remainder of this section, for $u,v\in  \widetilde W_0^{s,2}(\Omega)$, we shall let
\begin{align*}
\mathcal E(u,v):=\frac{C_{N,s}}{2}\int_{\RR^N}\int_{\RR^N}\frac{(u(x)-u(y))(v(x)-v(y))}{|x-y|^{N+2s}}\;dxdy.
\end{align*}

We start with the following theorem result which can be viewed as  the first main result of this section.

\begin{theorem}\label{Main1}
Let $f_0\in L^p(\Omega)$ and $f_1\in L^q(\Omega\times\Omega)$ with $p>\frac N{2s}$ and $q>\frac Ns$. Then there exists a unique function $u\in  \widetilde{W}_0^{s,2}(\Om)$ satisfying
\begin{align}\label{Eq-G1}
\mathcal E(u,v)=\int_{\Omega}f_0v\;dx+\int_{\Omega}\int_{\Omega}f_1(x,y) D_{s,2} v[x,y] \;dxdy,
\end{align}
for every $v\in   \widetilde{W}_0^{s,2}(\Om)$. In addition, $u\in L^\infty(\Omega)$ and there is a constant $C>0$ such that
\begin{align}\label{Eq-G2}
\|u\|_{L^\infty(\Om)}\le C\left(\|f_0\|_{L^p(\Omega)}+\|f_1\|_{L^q(\Om\times\Om)}\right).
\end{align}
\end{theorem}

To prove the theorem we need the following lemma which is of analytic nature and will be useful in deriving some a priori estimates of weak solutions of elliptic type equations (see e.g. \cite[Lemma B.1.]{kinderlehrer1980}). 

\begin{lemma}\label{lem-01}
Let $\Phi = \Phi(t)$ be a nonnegative, non-increasing function on a half line $t\ge k_0\ge 0$ such that there are positive constants $c, \alpha$ and $\delta$ ($\delta >1$) with
\begin{equation*}
\Phi(h) \le c(h-k)^{-\alpha}\Phi(k)^{\delta}\mbox{ for }  h>k\ge k_0.
\end{equation*}
Then
\begin{equation*}
\Phi(k_0+d) = 0\quad \mbox{ with }\quad d^{\alpha}= c \Phi(k_0)^{\delta -1}2^{\alpha\delta/(\delta -1)}.
\end{equation*}
\end{lemma}

\begin{proof}[\bf Proof of Theorem \ref{Main1}]
We prove the result in several steps.

{\bf Step 1}: Firstly, we show that there is a unique $u\in  \widetilde{W}_0^{s,2}(\Om)$ satisfying \eqref{Eq-G1}. It suffices to show that the right hand side of \eqref{Eq-G1} defines a continuous linear functional on $\widetilde{W}_0^{s,2}(\Om)$. Indeed, recall that $\widetilde{W}_0^{s,2}(\Om)\hookrightarrow L^{p'}(\Omega)$ by Remark \ref{remark}. Hence,
using this embedding and the classical H\"older inequality, we get that there is a constant $C>0$ such that 
\begin{align*}
&\left|\int_{\Omega}f_0v\;dx+\int_{\Omega}\int_{\Omega}f_1(x,y) D_{s,2} v[x,y] \;dxdy\right|\\
\le& \|f_0\|_{L^p(\Omega)}\|v\|_{L^{p'}(\Omega)}+\|f_1\|_{L^2(\Omega\times\Omega)}\|D_{2,s}v\|_{L^2(\Omega\times\Omega)}\\
\le&C \left(\|f_0\|_{L^p(\Omega)}      +\|f_1\|_{L^2(\Omega\times\Omega)} \right)  \|v\|_{\widetilde{W}_0^{s,2}(\Om)}.
\end{align*}
Since the bilinear form $\mathcal E$ is continuous and coercive,  it follows from the classical Lax-Milgram lemma that there is unique function $u\in \widetilde{W}_0^{s,2}(\Om)$ satisfying \eqref{Eq-G1}.

{\bf Step 2}: Notice that if $N<2s$, it follows from the embedding \eqref{inj1} that $u\in L^\infty(\Omega)$. We give the proof for the case $N>2s$. The case $N=2s$ follows with a simple modification of the case $N>2s$. Therefore, throughout the proof we assume that $N>2s$. 

{\bf Step 3}: Let $u\in \widetilde{W}_0^{s,2}(\Om)$ be the unique function satisfying \eqref{Eq-G1}. Let $k\ge 0$ be a real number and set $u_k:=(|u|-k)^+\mbox{sgn}(u)$. By \cite[Lemma 2.7]{War} we have that $u_k\in \widetilde{W}_0^{s,2}(\Om)$ for every $k\ge 0$. Proceeding exactly as in the proof of \cite[Theorem 2.9]{AnPfWa2017} (see also \cite[Proposition 3.10 and Section 3.3]{antil2017optimal}) we get that
\begin{align}\label{Eq-G3}
\mathcal E(u_k,u_k)\le \mathcal E(u_k,u) =\int_{\Omega}f_0u_k\;dx+\int_{\Omega}\int_{\Omega}f_1(x,y) D_{s,2} u_k[x,y] \;dxdy,
\end{align}
for every $k\ge 0$.

Let $A_k:=\{x\in\Omega:\;|u(x)|\ge k\}$. Then it is clear that
\begin{equation}\label{Eq-G4}
u_k=
\begin{cases}
(|u|-k)\mbox{sign}(u)\;\;\;&\mbox{ in }\; A_k\\
0&\mbox{ in }\;\Omega\setminus A_k.
\end{cases}
\end{equation}
Let $p_1\in [1,\infty]$ be such that
\begin{align*}
\frac{1}{p}+\frac{1}{2^\star}+\frac{1}{p_1}=1,
\end{align*}
where we recall that $2^\star:=\frac{2N}{N-2s}$.
Since by assumption $p>\frac{N}{2s}=\frac{2^\star}{2^\star-2}$, we have that
\begin{align}\label{Eq1}
\frac{1}{p_1}=1-\frac{1}{2^\star}-\frac{1}{p}=\frac{{2^\star}}{{2^\star}}-\frac{1}{2^\star}-\frac{1}{p}>\frac{{2^\star}}{{2^\star}}-\frac{1}{2^\star}-\frac{2^\star-2}{2^\star}=\frac{1}{2^\star}\Longrightarrow p_1<2^\star.
\end{align}
Using  \eqref{Eq-G4}, the continuous embedding  $ \widetilde{W}_0^{s,2}(\Om) \hookrightarrow L^{2^\star}(\Om)$, and the H\"older inequality, we get that there is a constant $C>0$ such that
\begin{align}\label{Eq-G5}
\int_{\Omega}f_0u_k\;dx=\int_{A_k}f_0u_k\;dx\le& \|f_0\|_{L^p(\Omega)}\|u_k\|_{L^{2^\star}(\Omega)}  \|\chi_{A_k}\|_{L^{p_1}(\Omega)}\notag\\
\le&\|f_0\|_{L^p(\Omega)}\|u_k\|_{ \widetilde{W}_0^{s,2}(\Om)}  \|\chi_{A_k}\|_{L^{p_1}(\Omega)},
\end{align}
fore every $k\ge 0$.

Let $\delta_1:=\frac{2^\star}{p_1}>1$ by \eqref{Eq1}. Using the H\"older inequality again, we get that there is a constant $C>0$ such that for every $k\ge 0$, we have
\begin{align}\label{Eq1-1}
 \|\chi_{A_k}\|_{L^{p_1}(\Omega)}\le C \|\chi_{A_k}\|_{L^{2^\star}(\Omega)}^{\delta_1}.
\end{align}

{\bf Step 4}: Next, let $q_1\in [1,\infty]$ be such that 
\begin{align*}
\frac{1}{q}+\frac{1}{2}+\frac{1}{q_1}=1.
\end{align*}
Since by assumption $q>\frac{N}{s}=2\frac{N}{2s}=2\frac{2^\star}{2^\star-2}$, we have that
\begin{align}\label{Eq-G6}
\frac{1}{q_1}=1-\frac{1}{2}-\frac{1}{q}=2\frac{2^\star}{2\cdot 2^\star}-\frac{1}{2}-\frac{1}{q}>\frac{{2\cdot 2^\star}}{{2\cdot 2^\star}}-\frac{1}{2}-\frac{2^\star-2}{2\cdot 2^\star}=\frac{1}{2^\star}\Longrightarrow q_1<2^\star.
\end{align}
Using \eqref{Eq-G4}, the continuous embedding $\widetilde{W}_0^{s,2}(\Om) \hookrightarrow L^{2^\star}(\Om)$, and the H\"older inequality again, we can deduce that there is a constant $C>0$ such that
\begin{align}\label{Eq-G7}
&\int_{\Omega}\int_{\Omega}f_1(x,y) D_{s,2} u_k[x,y] \;dxdy=\int_{A_k}\int_{A_k}f_1(x,y) D_{s,2} u_k[x,y] \;dxdy\notag\\
&+\int_{A_k}\int_{\Omega\setminus A_k}f_1(x,y) D_{s,2} u_k[x,y] \;dxdy+\int_{\Omega\setminus A_k}\int_{A_k}f_1(x,y) D_{s,2} u_k[x,y] \;dxdy\notag\\
\le&\|f_1\|_{L^q(\Omega\times\Omega)}\|D_{s,2}u_k\|_{L^2(\Omega\times\Omega)}\|\chi_{A_k\times A_k}\|_{L^{q_1}(\Omega\times\Omega)}\notag\\
&+2\|f_1\|_{L^q(\Omega\times\Omega)}\|D_{s,2}u_k\|_{L^2(\Omega\times\Omega)}\|\chi_{A_k\times (\Om\setminus A_k)}\|_{L^{q_1}(\Omega\times\Omega)}\notag\\
\le &C\|f_1\|_{L^q(\Omega\times\Omega)}\|u_k\|_{ \widetilde{W}_0^{s,2}(\Om)}  \|\chi_{A_k}\|_{L^{q_1}(\Omega)},
\end{align}
for every $k\ge 0$. 

Let $\delta_2:=\frac{2^\star}{q_1}>1$ by \eqref{Eq-G6}. Using the H\"older inequality we get that there is a constant $C>0$ such that for every $k\ge 0$, we have
\begin{align}\label{Eq1-2}
 \|\chi_{A_k}\|_{L^{q_1}(\Omega)}\le C \|\chi_{A_k}\|_{L^{2^\star}(\Omega)}^{\delta_2}.
\end{align}

{\bf Step 5}: Let $\delta:=\min\{\delta_1,\delta_2\}>1$. It follows from \eqref{Eq1-1} that there is a constant $C>0$ such that for every $k\ge 0$
\begin{align*}
 \|\chi_{A_k}\|_{L^{p_1}(\Omega)}\le C \|\chi_{A_k}\|_{L^{2^\star}(\Omega)}^{\delta_1}=  \|\chi_{A_k}\|_{L^{2^\star}(\Omega)}^{\delta} \|\chi_{A_k}\|_{L^{2^\star}(\Omega)}^{\delta_1-\delta}\le  \|\chi_{A_k}\|_{L^{2^\star}(\Omega)}^{\delta}.
 \end{align*}

Similarly, it follows from \eqref{Eq1-2} that there is a constant $C>0$ such that for every $k\ge0$
\begin{align*}
 \|\chi_{A_k}\|_{L^{q_1}(\Omega)}\le C \|\chi_{A_k}\|_{L^{2^\star}(\Omega)}^{\delta_2}=  \|\chi_{A_k}\|_{L^{2^\star}(\Omega)}^{\delta} \|\chi_{A_k}\|_{L^{2^\star}(\Omega)}^{\delta_2-\delta}\le  \|\chi_{A_k}\|_{L^{2^\star}(\Omega)}^{\delta}.
 \end{align*}

We have shown that that there is a constant $C>0$ such that 
\begin{align}\label{Eq1-3}
 \max\{\|\chi_{A_k}\|_{L^{p_1}(\Om)},  \|\chi_{A_k}\|_{L^{q_1}(\Omega)}\}\le C \|\chi_{A_k}\|_{L^{2^\star}(\Omega)}^{\delta},
\end{align}
for every $k\ge 0$.

Using \eqref{Eq-G3}, \eqref{Eq-G5}, \eqref{Eq-G7}, \eqref{Eq1-3} and the fact that there is a constant $C>0$ such that 
\begin{align*}
C\|u_k\|_{ \widetilde{W}_0^{s,2}(\Om)}\le \mathcal E(u_k,u_k),
\end{align*}
 we get that there is a constant $C>0$ such that for every $k\ge 0$, we have
\begin{align}\label{Eq-G8}
\|u_k\|_{ \widetilde{W}_0^{s,2}(\Om)}\le C\left(\|f_0\|_{L^p(\Omega)}+\|f_1\|_{L^q(\Omega\times\Omega)}\right)  \|\chi_{A_k}\|_{L^{2^\star}(\Omega)}^\delta.
\end{align}
Using the continuous embedding $\widetilde{W}_0^{s,2}(\Om) \hookrightarrow L^{2^\star}(\Om)$ and \eqref{Eq-G8}, we get that there is a constant $C>0$ such that for every $k\ge 0$, we have 
\begin{align}\label{Eq-G9}
\|u_k\|_{L^{2^\star}(\Om)}\le C\left(\|f_0\|_{L^p(\Omega)}+\|f_1\|_{L^q(\Omega\times\Omega)}\right)  \|\chi_{A_k}\|_{L^{2^\star}(\Omega)}^\delta.
\end{align}

{\bf Step 6}: Now let $h>k\ge 0$. Then $A_h\subset A_k$ and in $A_h$ we have that $|u_k|\ge (h-k)$. Thus, it follows from \eqref{Eq-G9} that there is a constant $C>0$ such that for every $h>k\ge 0$, 
\begin{align}\label{Eq-G10}
  \|\chi_{A_h}\|_{L^{2^\star}(\Omega)}\le C(h-k)^{-1}\left(\|f_0\|_{L^p(\Omega)}+\|f_1\|_{L^q(\Omega\times\Omega)}\right)  \|\chi_{A_k}\|_{L^{2^\star}(\Omega)}^\delta.
\end{align}
Let $\Phi(k):= \|\chi_{A_k}\|_{L^{2^\star}(\Om)}$. It follows from \eqref{Eq-G10} that

\begin{align*}
\Phi(h)\le C(h-k)^{-1}\left(\|f_0\|_{L^p(\Omega)}+\|f_1\|_{L^q(\Omega\times\Omega)}\right) \Phi(k),
\end{align*}
for all $h>k\ge 0$.
Finally, applying Lemma \ref{lem-01} to the function, $\Phi$, we can deduce that there is a constant $C_1>0$ such that
\begin{align*}
\Phi(K)=0\;\mbox{ with }\; K=C_1C\left(\|f_0\|_{L^p(\Omega)}+\|f_1\|_{L^q(\Omega\times\Omega)}\right).
\end{align*}
We have shown the estimate \eqref{Eq-G2} and the proof is finished.
\end{proof}

We also have the following regularity result for solutions to the problem \eqref{eq:Sd} which is the second main result of this section. 
Here, we reduce the regularity of datum $z$, compare  with Proposition~\ref{prop:Antil_Warma}. In view of Step 2 in the proof of Lemma~\ref{lem-01}, we shall focus on the case $N>2s$.

\begin{theorem}\label{Main2}
Let $\Omega\subset\RR^N$ be a bounded open set with a Lipschitz continuous boundary.
Let $0<t<s<1$, $2<\frac{N}{s}<p\le\infty$, $1 \le p'=\frac{p}{p-1} < 2$ and $\frac{1}{p'} < t < 1$. Then for every $z\in \widetilde{W}^{-t,p}(\Om)$, there is a unique solution $u\in \widetilde W_0^{s,2}(\Om)$ of the Dirichlet problem \eqref{eq:Sd}.
In addition, $u\in L^\infty(\Omega)$ and there is a constant $C>0$ such that
\begin{align}\label{M00}
\|u\|_{L^\infty(\Omega)}\le C\|z\|_{ \widetilde{W}^{-t,p}(\Om)}.
\end{align}
\end{theorem}

\begin{proof}
We prove the result in several steps.

{\bf Step 1}: Firstly, for $z\in \widetilde{W}^{-t,p}(\Om)$, by a solution to the Dirichlet problem \eqref{eq:Sd}, we mean a function $u\in \widetilde W_0^{s,2}(\bOm)$  satisfying 
\begin{align}\label{M0}
\mathcal E(u,v)=\langle z,v\rangle_{\widetilde{W}^{-t,p}(\Om),\widetilde{W}^{t,p'}(\Om)}, \;\;\;\forall\; v\in \widetilde W_0^{t,p'}(\Om),
\end{align}
provided that the left and right hand sides expressions make sense.

{\bf Step 2}: Secondly, since $\frac{1}{p'} < t < 1$ and $z\in \widetilde{W}^{-t,p}(\Om)$, it follows from Corollary \ref{cor-dual} that there exists a pair of functions $(f^0,f^1)\in L^p(\Omega)\times L^p(\Omega\times\Omega)$ such that
\begin{align}\label{M1}
       \langle z,v\rangle_{\widetilde{W}^{-t,p}(\Om),\widetilde{W}^{t,p'}(\Om)}
            = \int_{\Omega} f^0 v \, dx + \int_{\Om} \int_{\Om} f^1 D_{t,p'} v[x,y] 
                    \; dx \; dy, \quad \forall v\in \widetilde{W}^{t,p'}_0(\Om) .
\end{align}

Choose $(f^0,f^1)\in L^p(\Omega)\times L^p(\Omega\times\Omega)$ satisfying \eqref{M1} and are such that
\begin{align}\label{norm}
\|z\|_{\widetilde{W}^{-t,p}(\Om)}=\|f^0\|_{L^p(\Omega)} +\|f^1\|_{L^p(\Omega\times\Omega)}.
\end{align}

Since $0<t<s<1$ and $2>p'$, it follows from \cite[Proposition 2.2]{antil2018b} that we have the continuous embedding $ \widetilde{W}^{s,2}_0(\Om) \hookrightarrow  \widetilde{W}^{t,p'}_0(\Om)$. More precisely, there is a constant $C>0$ such that

\begin{align*}
\left| D_{t,p'} v[x,y] \right|=&\frac{|v(x)-v(y)|}{|x-y|^{\frac{p'}{N}+t}}=\frac{|v(x)-v(y)|}{|x-y|^{\frac{2}{N}+s}}\frac{1}{|x-y|^{\frac{p'-2}{N}+t-s}}=\frac{|v(x)-v(y)|}{|x-y|^{\frac{2}{N}+s}}|x-y|^{s-t+\frac{2-p'}{N}}\\
\le &C\frac{|v(x)-v(y)|}{|x-y|^{\frac{2}{N}+s}}=C\left| D_{s,2} v[x,y] \right|,
\end{align*}
where we have used that $s-t+\frac{2-p'}{N}>0$. Thus, for every $v\in \widetilde{W}^{s,2}_0(\Om)$, we have 
\begin{align*}
\|D_{t,p'} v\|_{L^{p'}(\Omega\times\Omega)}\le C\|D_{s.,2} v\|_{L^2(\Omega\times\Omega)}.
\end{align*}
Hence, \eqref{M1} also holds for every $v\in  \widetilde{W}^{s,2}_0(\Om)$. Thus the right and left hand sides of \eqref{M0} make sense.

{\bf Step 3}: We claim that there is a unique $u\in \widetilde W_0^{s,2}(\Om)$ satisfying \eqref{M0}. As in the proof of Theorem \ref{Main1}, we have to show that the right hand side of \eqref{M1} defines a linear continuous functional on $\widetilde W_0^{s,2}(\Om)$. Indeed, let $v\in \widetilde W_0^{s,2}(\Om)$. Using Step 1 and Remark \ref{remark}, we get that there is a constant $C>0$ such that
\begin{align*}
 \left|\langle z,v\rangle_{\widetilde{W}^{-t,p}(\Om),\widetilde{W}^{t,p'}(\Om)} \right|
            =&\left| \int_{\Omega} f^0 v \, dx + \int_{\Om} \int_{\Om} f^1 D_{t,p'} v[x,y] 
                    \; dx \; dy\right| , \quad \forall v\in \widetilde{W}^{t,p'}_0(\Om) \\
             \le &\|f^0\|_{L^p(\Omega)}\|v\|_{L^{p'}(\Omega)} +\|f^1\|_{L^p(\Omega\times\Omega)}\|D_{t,p'} v\|_{L^{p'}(\Omega\times\Omega)}\\
             \le &C\left(\|f^0\|_{L^p(\Omega)}\|v\|_{L^{2}(\Omega)} +\|f^1\|_{L^p(\Omega\times\Omega)}\|D_{t,p'} v\|_{L^{2}(\Omega\times\Omega)}\right)\\
        =&C\left(\|f^0\|_{L^p(\Omega)} +\|f^1\|_{L^p(\Omega\times\Omega)}\right)\|v\|_{\widetilde{W}^{s,2}_0(\Om)},
\end{align*}
and the claim is proved.

{\bf Step 4}:  It follows from Step 3 that the unique $u\in \widetilde{W}^{s,2}_0(\Om)$ satisfying \eqref{M0} is such that
\begin{align*}
\mathcal E(u,v)= \langle z,v\rangle_{\widetilde{W}^{-t,p}(\Om),\widetilde{W}^{t,p'}(\Om)}\le C\int_{\Omega}|f^0v|\;dx+ \int_{\Om} \int_{\Om}\left| f^1 D_{s,2} v[x,y] \right|
                    \; dx \; dy.
\end{align*}
Therefore, proceeding exactly as in the proof of Theorem \ref{Main1}, we get that $u\in L^\infty(\Omega)$ and there is a constant $C>0$ such that
\begin{align*}
\|u\|_{L^\infty(\Om)}\le C\left(\|f^0\|_{L^p(\Omega)} +\|f^1\|_{L^p(\Omega\times\Omega)}\right)=C\|z\|_{\widetilde{W}^{-t,p}(\Om)},
\end{align*}
where we have used \eqref{norm}.
We have shown the estimate \eqref{M00} and the proof is finished.
\end{proof}

We have the following regularity result as a corollary of Theorems \ref{Main2} and \ref{pbound}.

\begin{corollary}\label{cor:Main2}
Let $\Omega\subset\RR^N$ be a bounded Lipschitz domain satisfying the exterior cone condition.
Let $0<t<s<1$, $2<\frac{N}{s}<p\le\infty$, $1 \le p'=\frac{p}{p-1} < 2$ and $\frac{1}{p'} < t < 1$. Let $z\in \widetilde{W}^{-t,p}(\Om)$ and let $u\in \widetilde W_0^{s,2}(\Om)$ be the unique solution of \eqref{eq:Sd}. Then $u\in C_0(\Omega)$.
\end{corollary}

\begin{proof}
Let $z\in \widetilde{W}^{-t,p}(\Om)$ and $\{z_n\}_{n\ge 1}\subset L^\infty(\Om)$  a sequence such that $z_n\to z$ in $\widetilde{W}^{-t,p}(\Om)$ as $n\to\infty$. Let $u_n\in  \widetilde{W}^{s,2}_0(\Om)$ satisfy
\begin{align*}
\mathcal E(u_n,v)= \langle z_n,v\rangle_{\widetilde{W}^{-t,p}(\Om),\widetilde{W}^{t,p'}(\Om)} =\int_{\Om}z_nv\;dx, \;\;\;\forall\; v\in \widetilde W_0^{s,2}(\Om).
\end{align*}
It follows from Theorem \ref{pbound} that $u_n\in C_0(\Omega)$. Since $u_n-u\in  \widetilde W_0^{s,2}(\bOm)$ satisfies
\begin{align*}
\mathcal E(u_n-u,v)= \langle z_n-z,v\rangle_{\widetilde{W}^{-t,p}(\Om),\widetilde{W}^{t,p'}(\Om)} \;\;\;\forall\; v\in \widetilde W_0^{s,2}(\Om),
\end{align*}
it follows from Theorem \ref{Main2} that $u_n-u\in L^\infty(\Omega)$ and there is a constant $C>0$ (independent of $n$) such that
\begin{align*}
\|u_n-u\|_{L^\infty(\Omega)}\le C\|z_n-z\|_{ \widetilde{W}^{-t,p}(\Om)}.
\end{align*}
Since $u_n\in C_0(\Om)$ and $z_n\to z$ in $\widetilde{W}^{-t,p}(\Om)$ as $n\to\infty$, it follows from the preceding estimate that $u_n\to u$ in $L^\infty(\Omega)$ as $n\to\infty$. Thus, $u\in C_0(\Omega)$ and the proof is finished.
\end{proof}

Next we improve the regularity of $u$ solving \eqref{eq:Sda} with measure $\mu$ as the 
right-hand-side. Notice that such a result will immediately improve the regularity of
the adjoint variable $\bar{\xi}$ solving \eqref{eq:b}. Recall that the best result so far proved
for the solution to \eqref{eq:Sda} is given in Theorem~\ref{thm:mesdata}. 
    \begin{corollary}\label{cor:Main3}
        Let $\Omega\subset\RR^N$ be a bounded Lipschitz domain 
        satisfying the exterior cone condition.
        Let $0<t<s<1$, $2<\frac{N}{s}<p\le\infty$, $1 \le  p'=\frac{p}{p-1} < 2$ and 
        $\frac{1}{p'} < t < 1$. Let $\mu\in \mathcal{M}(\Om)$. Then there is a unique 
        solution $u \in \widetilde{W}_0^{t,p'}(\Om)$ to \eqref{eq:Sda} and there is a constant 
        $C > 0$ such that
        \[
            \|u\|_{\widetilde{W}_0^{t,p'}(\Om)} \le C \|\mu\|_{\mathcal{M}(\Om)} .
        \]
    \end{corollary}
    
    \begin{proof}
        The proof follows exactly as the proof of Theorem~\ref{thm:mesdata} with the exception that for the 
        inequality \eqref{eq:imp}, we use Corollary~\ref{cor:Main2} to arrive at 
        \[
            \left| \int_\Om u\xi \;dx \right| 
                \le \|\mu\|_{\mathcal{M}(\Om)} \|v\|_{C_0(\Om)}
                \le C \|\mu\|_{\mathcal{M}(\Om)} \|\xi\|_{\widetilde{W}^{-t,p}(\Om)} .
        \]
The preceding estimate implies that $u \in (\widetilde{W}^{-t,p}(\Om))^\star = \widetilde{W}_0^{t,p'}(\Om)$. The proof is finished.
    \end{proof}
Recall that the ``strong form" of the adjoint equation \eqref{eq:b} is given by     
    \begin{equation}\label{eq:Sdadd}
     \begin{cases}
        (-\Delta)^s \bar{\xi} = J_u(\bar{u},\bar{z}) + \bar\mu \quad \;\;\; &\mbox{ in } \Om, \\
                    \bar{\xi} = 0 \quad &\mbox{cin } \RR^N\setminus \Om  .
     \end{cases}                
    \end{equation}
Then using Corollary~\ref{cor:Main3} and the fact that $J_u(\bar{u},\bar{z}) \in L^2(\Om)$, we obtain the following regularity result for the adjoint variable $\bar{\xi}$. 
    \begin{corollary}[\bf Regularity of the adjoint variable]
    \label{cor:Main4}
        Let $\bar{\mu}\in \mathcal{M}(\Om)$ and let $\bar{\xi}$ be the Lagrange multiplier given in 
        Theorem~\ref{thm:optcond}. Then under the conditions of Corollary~\ref{cor:Main3}, 
        we have that $\bar{\xi} \in \widetilde{W}_0^{t,p'}(\Om)$. 
    \end{corollary}
We conclude the paper with the following remark.

\begin{remark}[\bf Regularity for the controls]
In case of the widely used cost functional 
$$J(u,z) := \frac{1}{2}\|u-u_d\|^2_{L^2(\Om)} + \frac{\alpha}{2} \|z\|^2_{L^2(\Om)},$$
 where 
$u_d \in L^2(\Om)$ is the given datum
from \eqref{eq:c}, invoking the standard projection formula type approach  (see e.g., \cite{FTroeltzsch_2010a}), it is possible to show that $\bar{z}$ has the same regularity as $\bar{\xi}$. 
\end{remark}

\bibliographystyle{plain}
\bibliography{refs}

\end{document}